\documentclass[11pt]{article}
\usepackage{amsfonts, amsmath, amssymb, amsthm, enumitem, etoolbox, wasysym, graphicx, dirtytalk, titlesec, booktabs}
\usepackage[font=small,labelfont=bf]{caption}
\usepackage[backend=biber, style=numeric , giveninits=true, doi=false,isbn=false,url=false, sorting=nyt]{biblatex}
\usepackage[margin=1.0in]{geometry}
\usepackage{footnote}

\theoremstyle{definition}
\newtheorem{thm}{Theorem}[section]
\newtheorem{lem}[thm]{Lemma}

\newtheorem{cor}[thm]{Corollary}

\newtheorem*{conj}{Conjecture}
\newtheorem*{rem}{Remark}

\titleformat{\section}
  {\center\scshape}{\thesection.}{0.5em}{}
  \titleformat{\subsection}
  {\center\scshape}{\thesubsection.}{0.5em}{}
\renewcommand{\O}{\mathcal{O}}
\newcommand{\p}{\mathfrak{p}}
\newcommand{\w}{\omega}
\renewcommand{\a}{\alpha}
\newcommand{\dia}[1]{\langle #1 \rangle}
\newcommand{\LegSym}[2]{\left(\frac{#1}{#2}\right)}
\newcommand{\logp}{\log_p}
\newcommand{\psim}{\psi_{m}}
\newcommand{\lcm}{\operatorname{lcm}}
\newcommand{\z}{\zeta_p^*}
\newcommand{\eps}{\varepsilon}
\newcommand{\Q}{\mathbb{Q}}

\newcommand{\Z}{\mathbb{Z}}
\newcommand{\C}{\mathbb{C}}
\newcommand{\F}{\mathcal{F}}
\renewcommand{\mod}[1]{\left(\bmod{\,#1}\right)}
\newcommand{\s}{\mathcal{S}}
\newcommand{\aac}{\text{Ankeny--Artin--Chowla conjecture}}

\begin{document}
\begin{center}
{\bf \uppercase{A note on arithmetic congruences}}
\vskip 20pt
{\textsc{Nic Fellini}}
% \\
% {\smallit 
% % Department of Mathematics and Statistics\\ 
% Queen's University}
% {\tt n.fellini@queensu.ca}\\ 
\end{center}
\vskip 20pt

% \centerline{
\noindent
{\small \textsc{Abstract.} 
By analyzing the coefficients of the power series defining the Kubota--Leopoldt $p$-adic $L$-function associated to the non-trivial character of a real quadratic field, we prove a congruence of Ankeny--Artin--Chowla-type for prime power modulus. Additionally, we show how some classical congruences relating Bernoulli numbers and Wilson quotients fit naturally into the theory of the $p$-adic Riemann zeta function. 

\section{Introduction}
The most widely known and implemented approach to public-key cryptography, the RSA scheme, is based on the simple philosophy that factoring large integers is computationally very challenging. Many of the most successful factoring algorithms (notably, the number field sieve \cite{Stevenhagen2008}) rely on moving the question of factorizations over the integers to one of factorizations over larger number systems, the so-called \textit{ring of integers} $\O_K$ of a number field $K$. A \textit{number field} is a finite degree extension of the rational numbers. Motivated by the relationship between integers and rational numbers, $\O_K$ serves as the generalization of the integers in the context of number fields. However, unlike the integers, $\O_K$ does not necessarily admit unique factorization of elements. The failure of $\O_K$ to be a unique factorization domain (UFD) is captured by the \textit{ideal class group} $\operatorname{Cl}(K)$ of $\O_K$, and in particular, by the cardinality or \textit{class number}, of this (finite) group. When the class number is one, the ring of integers is a principal ideal domain (PID) and hence a UFD. For a synoptic overview of modern factorization methods in cryptoanalysis we invite the reader to consult \cite{Boudot2022} and the references therein as well as the fantastic book \cite{Cohen1993} for many of the mathematical details. \newline 

Given their prominent place in algebraic number theory, the computation of ideal class groups and class numbers is of both practical and theoretical interest. For quadratic fields, there are excellent analytic methods for computing class numbers \cite{Louboutin2002}. Beside the analytic methods, there are also arithmetic methods that can be used to compute the class number for many quadratics fields \cite{Ankeny1952, Carlitz1953b, Fellini2025}. We describe results in this direction now. \newline

In 1951, N. C. Ankeny, E. Artin, and S. Chowla announced  four congruence relations involving the arithmetic invariants of real quadratic fields $\Q(\sqrt{d})$ and (generalized) Bernoulli numbers \cite{Ankeny1951}. Three of these relations were proved in their 1952  article \cite{Ankeny1952}  while the fourth was proved by L. Carlitz the following year \cite{Carlitz1953a}. To animate our discussion, we recall the most well-known of these relations.

\begin{thm}[Ankeny--Artin--Chowla, 1952]
    Let $p\equiv 1\mod{4}$ be a prime number and $\frac{1}{2}\left(t+u\sqrt{p} \right)$ be the fundamental unit of $\Q(\sqrt{p})$. Set $r=\frac{p-1}{2}$.  Then
    \begin{equation}\label{eq: AAC congruence}
         \frac{2hu}{t}\equiv -\frac{B_r}{r}\mod{p}
    \end{equation}
    where $h$ is the class number of $\Q(\sqrt{p})$ and $B_{r}$ is the $r$-th Bernoulli number defined by the relation
    \[
    \frac{x}{e^x-1} = \sum_{n=0}^\infty B_n\frac{x^n}{n!}.
    \]
\end{thm}
Philosophically, (\ref{eq: AAC congruence}) provides ``local" information about a critical value\footnote{In the context of this work, an integer $m$ is a \textit{critical point} of a Dirichlet $L$-function if neither $m$ nor $1-m$ is a pole of the $\Gamma$-factor of $L(s, \chi)$.  The value $L(m,\chi)$ is called a \textit{critical value} of $L(s,\chi)$.} of one $L$-function in terms of a critical value of another (suitably twisted) $L$-function. The left-hand side of (\ref{eq: AAC congruence}) should be viewed as a $p$-adic approximation of the  Dirichlet $L$-function \[L(s, \chi_p) = \sum_{n\geq 1} \frac{\chi_p(n)}{n^s}\] at $s=1$, where $\chi_p$ is the non-trivial character of $\Q(\sqrt{p})$. A standard application of the analytic class number formula (see, for example, Exercise 10.5.12 of \cite{MurtyProbsAlg}) links the value $L_p(1, \chi_p)$ to the interesting ``global" arithmetic data of $\Q(\sqrt{p})$ that we see on the left hand side of the congruence. The right hand side is the value of the Riemann zeta function at $1-r$. The congruence between these quantities is mediated by the analytic and interpolation properties of the Kubota--Leopoldt $p$-adic $L$-function associated to the quadratic character $\chi_p$. \newline

It was seemingly known to Ankeny, Artin, and Chowla that when $p\equiv 5\mod{8}$, the class number $h$ of $\Q(\sqrt{p})$ is strictly less than $p$. This fact existed for quite some time as part of the number theoretic folklore. Indeed, in their 1960 paper \cite{AC} Ankeny and Chowla remark that they were aware of this at the time of the 1952 paper with Artin. Using this inequality together with the elementary fact that $p\nmid t$, they computed for primes $p\equiv 5 \mod{8}$ less than $2000$, that $p\nmid u$. Perhaps influenced by these calculations, they posed the following question:
\newline 

\textbf{Question. }\textit{Given a prime $p\equiv 1 \mod{4}$ and fundamental unit $\eps = \frac{1}{2}\left( t+ u\sqrt{p} \right)$ of $\Q(\sqrt{p})$, does $p\nmid u$? }\newline 

Over time, the affirmative answer to this question, that $p\nmid u$, has been dubbed the \textit{Ankeny--Artin--Chowla} conjecture. It has garnered considerable and varied attention in the literature (see, for example: \cite{Agoh2016, aktas, Benmerieme2024, Cohen2020, Sheingorn1989, Walsh2025}). The significance of the $\aac$ arises from the \textit{arithmetic congruence} in (\ref{eq: AAC congruence}). Since the class number $h$ of the real quadratic field $\Q(\sqrt{p})$ is strictly less than $p$ for any odd prime $p$ \cite{AC, Le1994}, and $t$ is coprime to $p$, if $p\nmid u$ the congruence in (\ref{eq: AAC congruence}) gives an arithmetic way of computing the class number. Indeed, when $p\nmid u$, the reduced residue of $\left(\frac{t}{u}\right)B_{\frac{p-1}{2}}\mod{p}$ is \textit{precisely} the class number of $\Q(\sqrt{p})$. The subject of this work is to provide a method for computing the class number when $p$ divides $u$. \newline

A decade after Ankeny, Artin, and Chowla's work, L. J. Mordell \cite{Mordell1961} made a similar conjecture for primes $p\equiv 3\mod{4}$, but did not obtain any congruences involving the arithmetic invariants of $\Q(\sqrt{p})$. We will call this conjecture of Mordell along with the $\aac$ the \textit{prime Ankeny--Artin--Chowla and Mordell} conjecture, or AACM for short.  Recently, A. Reinhart \cite{Reinhart2024a, Reinhart2024b} has found counterexamples to the AACM conjecture for primes $p\equiv 1 \mod{4}$ and $p\equiv 3\mod{4}$. Before the discovery of these counterexamples, L. Washington argued heuristically that the number of counterexamples to the AACM conjecture less than $x$ should be asymptotically $\log\log x$ \cite{Washington}. Given this, we expect the collection of all counterexamples to form a very sparse set among the primes.  Later still, several authors investigated a composite version of the AACM conjecture \cite{Mollin1986, Stephens1988, Yu1998}. When there is no risk of confusion, we will refer to both the composite and prime AACM conjectures simply as the AACM conjecture.  On similar heuristic grounds, we would expect the number of counterexamples to the composite AACM conjecture less than $x$ to behave asymptotically like $O(\log x)$ \cite{Fellini2025}. In both the prime and composite case, the available numerical data align quite closely with the heuristics \cite{Reinhart2024a, Reinhart2024b}.   \newline 

Although, we expect both of these conjectures to be false infinitely often, we would heuristically expect the following weaker version of the AACM conjecture: 
\begin{conj}[The weak AACM conjecture]
    Suppose $d>0$ is a square-free integer and write $d=pm$ for an odd prime $p$ and integer $m$. If $d\equiv 1\mod 4$, set $\delta =1$ otherwise let $\delta= 2$.  Let $\eps = \frac{\delta}{2}\left(t + u\sqrt{d}\right)$ be the fundamental unit of $\Q(\sqrt{d})$. Then there exists an integer $\kappa>1$, independent of $d$, such that $d^\kappa \nmid u$.
\end{conj}
Among the 23 known real quadratic fields for which $d\mid u$,  there are none for which $d^2\mid u$. However, there are three instances where an odd prime divisor $p$ of $d$ divides $u$ to order at least 2.  These are summarized in Table 1. \pagebreak
\begin{table}[h!] \label{tab: table 1}
    \centering
    \begin{tabular}{c c c r} \toprule 
     $d$  & Factorization & Class number of $\Q(\sqrt{d})$ & $v_p(u)$   \\
    \midrule
     $4099215$& $3\cdot 5\cdot 273281 $ & $4$ & $v_3(u)=3$ \\
     $125854178626$ & $2\cdot 11\cdot 17 \cdot 336508499$ & $8$ & $v_{11}(u)=2$\\
      $20256129307923$ & $3\cdot 569 \cdot 2659 \cdot 4462771$ & $16$ & $v_3(u)=2$  \\
      \bottomrule
    \end{tabular}
    \caption{The last column contains the $p$-adic valuation of $u$ for the prime divisors $p$ of $d$ that occur to at least the second power. }
\end{table} 

In retrospect, many essential ideas in the theory of cyclotomic fields, Iwasawa theory, and more generally, $p$-adic analytic number theory, can be found in the work of Ankeny, Artin, and Chowla. In particular, their research laid some of the early groundwork towards the development of the $p$-adic logarithm (see \cite{Fellini2024} for a relevant discussion of this development) and hints at the larger theory of $p$-adic $L$-functions as developed by T. Kubota and H.-W Leopoldt \cite{Kubota1964}, and K. Iwasawa \cite{Iwasawa1969}. \newline

The present work establishes a supercongruence of Ankeny--Artin--Chowla type, i.e., a congruence involving the arithmetic invariants of $\Q(\sqrt{d})$ and certain linear combinations of Bernoulli numbers $\mod{p^2}$. Under the weak AACM conjecture for $\kappa =2$, this yields an arithmetic method for computing the class number of many real quadratic fields. Notably, it can be used to compute the class number of all of the known counterexamples to the AACM conjecture.  We remark that while several previous attempts have been made to obtain such supercongruences---most notably by S. Jakubec in collaboration with M. La\v s\v s\'ak and F. Marko \cite{Jakubec1996, Jakubec1998a, Jakubec1998b, Jakubec2013}---the existing results impose significant restrictions on the primes considered.  We overcome these restrictions using methodology introduced by L. Washington \cite{Washington} and further developed by the author in \cite{Fellini2025}. However, we restrict our attention to the case that $p>3$. In the case of square-free $d$ of the form $d=3m$, reflection theorems (such as that of A. Scholz \cite{Scholz1932}) provide significantly stronger results on the 3-rank of the ideal class group of $\Q(\sqrt{3m})$ in terms of the $3$-rank of $\Q(\sqrt{-m})$. We refer the interested reader to chapter 10.2 of L. Washington's book \cite{Washington} or the excellent papers \cite{Lemmermeyer2005} and \cite{Ellenberg2007} for more on reflection theorems and their applications.

\section{Statement of results}
Let $\chi$ be a Dirichlet character of conductor $f$ and set
    \[
        F_{\chi}(t) = \sum_{a=1}^f \frac{\chi(a)te^{at}}{e^{ft}-1}.
    \]
Expanding $F_{\chi}(t)$ as formal power series we have:
\[
    F_{\chi}(t) = \sum_{n=0}^\infty B_{n, \chi} \frac{t^n}{n!}.
\]
We call the coefficients $B_{n, \chi}$ of $F_{\chi}(t)$ \textit{generalized Bernoulli numbers}.
\newline

Let $p>3$ be an odd prime and write $d=pm$ for a square-free positive integer $m$ coprime to $p$. If $d\equiv 1\mod{4}$ we set $\delta =1$ and if $d\equiv 2, 3\mod{4}$ we set $\delta =2$. Let $D = \delta^2d$ be the discriminant of the real quadratic field $\Q(\sqrt{d})$, $\chi_D(\cdot)$ denote the non-trivial character of $\Q(\sqrt{d})$ of conductor $D$. Finally, we define the primitive quadratic character $\psim(\cdot)$ of conductor $\delta^2m$ by  $\chi_D(\cdot) = \LegSym{\cdot }{p}\psim(\cdot)$ where $\LegSym{\cdot}{p}$ is the Legendre symbol $\mod{p}$.

\begin{thm}\label{thm: thm 1}
    Suppose $p>3$ is prime and write $d=pm$ for $m$ square-free and coprime to $p$. Further suppose that $d>5$. Denote the class number and fundamental unit of $\Q(\sqrt{d})$ by $h$ and $\eps = \frac{\delta}{2}\left(t+ u\sqrt{d} \right)$, respectively. Let $r=\frac{p-1}{2}$. Then, 
    \[
      \frac{4h}{\delta} \left( \frac{u}{t} + \frac{d}{3}\left( \frac{u}{t} \right)^3\right) \equiv -3(1-\psim(p)p^{r-1})\frac{B_{r, \psim}}{r} +\frac{B_{3r, \psim}}{3r} \mod{p^2}. 
    \]
    
\end{thm}
We remark that Theorem \ref{thm: thm 1} remains true if one considers $d=5$. However, there are extra terms that arise in this case (see Section \ref{sec: congruences}). Additionally,  the methods involved in the proof of this theorem easily extend to obtain congruences $\mod{p^3}$ if one is willing to exclude small primes (say $p=3,5, 7$). \newline

If $p$ violates the AACM conjecture, but satisfies the weak AACM  conjecture with $\kappa =2$, we obtain
\begin{cor}
Suppose that $p>5$ is an odd prime and that $p$ exactly divides $u$.  Then, 
    \[
    \frac{2h}{\delta} \left(\frac{u}{pt}\right) \equiv \frac{1}{p}\left( 3B_{r, \psi_m} - \frac{B_{3r, \psim}}{3}\right) \mod{p}. 
    \]
\end{cor}

We will call a prime $p$ a \textit{super-AACM prime for $d$} if $p\mid d$ and $p^m\mid u$ for some $m\geq 2$.
\begin{cor}
    Suppose $p>5$ is a super-AACM prime for a positive square-free integer $d$. Then 
    \[
    9B_{r, \psim} \equiv B_{3r, \psim} \mod{p^2}.
    \]
\end{cor}
In particular, the contrapositive of this statement gives a criteria to determine if a prime $p$ is a super-AACM prime for $d$. \newline

We also use this as an opportunity to show how a classical result of  N. G. W. H. Beeger \cite{Beeger1913} and independently E. Lehmer \cite{Lehmer1938}, relating the Bernoulli number $B_{p-1}$ to the \textit{Wilson quotient} $W_p=\frac{(p-1)!+1}{p}$, fits naturally into the theory of the $p$-adic Riemann zeta function. In fact, one can view their result as an analogue to Theorem 4 of \cite{Ankeny1951} and Theorem 1.5 of \cite{Fellini2025} for the ``trivial" quadratic field, $\Q$.  We recall Lehmer's more general result here.   
\begin{thm}[Lehmer, 1938] \label{thm: thm 2}
    Let $p\geq 3$ and $k$ be any positive integer. Then, 
    \[
    B_{k(p-1)} +\frac{1}{p} -1\equiv  kW_p \mod{p}.
    \]    
\end{thm}
Setting $k=1$ and $2$ then subtracting the resulting expressions we obtain the following congruences of Lehmer. 
\begin{cor}[Lehmer, 1938]
    Let $p\geq 3$. Then
    \[
    B_{2(p-1)}-B_{p-1} \equiv W_p \mod{p}. 
    \]
\end{cor}
In addition to the above two results our method yields the following new congruence. 
\begin{thm}\label{thm: thm 3}
    Let $p>5$ and set $R=1-p^{-1}$. Then for any positive integer $k$, 
    \[
      k(p-1)W_p\left(1 + \frac{pW_p}{2} \right) \equiv -B_{k(p-1)} +R  +k(B_{2(p-1)} - B_{p-1}) - \frac{k}{2}(B_{2(p-1)}-R) \mod{p^2}.
    \]
\end{thm}
We call a prime $p$ a \textit{super-Wilson prime} if $p^2\mid W_p$. With $k=1$ in the previous theorem, we deduce that:
\begin{cor}
    If a prime $p>3$ is a super-Wilson prime, then  
    \begin{equation}\label{eq: super-Wilson}
        4(B_{p-1} -R) \equiv B_{2(p-1)} -R \mod{p^2}.
    \end{equation}   
\end{cor}
One can easily check with the help of a computer that none of the known Wilson primes ($p=5, 13, 563$) are super-Wilson primes. 

\section{$p$-adic preliminaries}
Let $p$ be an odd prime. Denote by $\Z_p$, $\Q_p,$ and $\C_p$, the $p$-adic integers, $p$-adic rationals, and the completion of the algebraic closure of $\Q_p$, respectively. The usual $p$-adic valuation $v_p(\cdot)$ and $p$-adic absolute value $|\cdot|_p$ defined over $\Q$ can be extended to valuations and metrics on $\C_p$. We normalize the extended absolute value so that $|p|_p= p^{-1}$.  \newline 

We choose a fixed but arbitrary embedding of $\overline{\Q}$ into $\C_p$. We will say that an element $\a\in \C_p$ is \textit{$p$-integral} (with respect to this embedding) provided $|\a|_p\leq 1$. If $\alpha, \beta, \gamma \in \C_p$ and $\gamma\neq 0$, we will write $\alpha \equiv \beta \mod{\gamma}$ if the quotient $(\alpha -\beta)/\gamma$ is $p$-integral. \newline

Let $\chi$ be a Dirichlet character with conductor $f$. We will use the convention that $\chi(a)=0$ for all integers $a$ with $\gcd(a, f)>1$ and if $\chi=\chi_0$ is the principal character ($f=1$), then $\chi(a)=1$ for all integers $a$. The generalized Bernoulli numbers $B_{n,\chi}$ are defined by the relation
\[
\sum_{n=0}^\infty B_{n,\chi} \frac{t^n}{n!} = \sum_{a=1}^f \frac{\chi(a)te^{at}}{e^{ft}-1}. 
\]
We note that the $B_{n,\chi}$ are algebraic and valued in the field $\Q(\chi)$ generated by the values of $\chi$. Under our fixed embedding $\overline{\Q} \hookrightarrow \C_p$ we can view the values of $\chi$, and hence the $B_{n, \chi}$, as elements in $\C_p$.  \newline

Assume that $\chi$ is an even Dirichlet character. Then the Kubota--Leopoldt $L$-function is the unique $p$-adic meromorphic function
\[
L_p(s, \chi) : \Z_p \to \C_p
\]
such that 
\[
L_p(1-n, \chi) =  -  (1-\chi\w^{-n}(p)p^{n-1}) \frac{B_{n,\chi}}{n}
\]
for all integers $n\geq 1$. Here, $\w$ is the $p$-adic Teichm\"{u}ller character which is characterized as follows: for any $\a\in \Z_p^\times$,  $\w(a)$ is the unique $(p-1)$-st root of unity in $\Z_p$ satisfying $\w(a)\equiv a\mod{p}$. One can easily check that $\w$ defines a character with conductor $p$ and that every $a\in \Z_p^\times$ can be uniquely written as $a=\w(a)\dia{a}$ where $\dia{a}\equiv 1 \mod{p}$. The $p$-adic $L$-function is analytic unless $\chi=\chi_0$ is the principal character. In this case, it has a simple pole at $s=1$ with residue $R= 1-p^{-1}$. We note that if $n\equiv 0 \mod{p-1}$ then $L_p(1-n, \chi)$ is the value of the Dirichlet $L$-function $L(1-n, \chi)$ with the Euler factor at $p$ removed. \newline 

We recall that the non-zero elements of $\C_p$ can be decomposed as 
\[
\C_p^\times  = p^{\Q}\times W \times U_1
\]
where $p^\Q$ is the set $\{p^r : r\in \Q\}$, $W$ is the set of roots of unity with order coprime to $p$, and $U_1 = \{x\in \C_p : |x-1|_p<1\}$ (see Proposition 5.4 of \cite{Washington} or section III.4 of \cite{Koblitz} for details). The \textit{$p$-adic logarithm }is defined by the formal power series
    \[
    -\log_p(1-x) = \sum_{n=1}^\infty \frac{x^n}{n}.
    \]
    This series converges $p$-adically for all $x\in U_1$. Moreover, it can be (\textit{uniquely}) extended to all of $\C_p^\times$ by declaring that $\log_p(p)=\log_p(w)=0$ for any root of unity $w\in W$. Moreover, this extension of the $p$-adic logarithm satisfies the familiar property that 
    \[
    \log_p(\alpha \beta) = \log_p(\alpha) + \logp(\beta)
    \]
    for all $\alpha, \beta \in \C_p^\times$. Further details can be found in \cite[Proposition 5.4]{Washington}. \newline

For square-free positive integers of the form $d=pm$ where $p$ is an odd prime, it turns out that the $p$-adic logarithm of the fundamental unit of $\Q(\sqrt{d})$ admits a pleasantly simple expression in $\C_p$. We record this here for later use and refer the reader to Proposition 2.3 of \cite{Fellini2025} for a proof. 
\begin{lem}\label{prop: p-adic log of fund. unit}
    Fix an odd prime $p$ and let $d=pm$ be a positive square-free integer. Denote the fundamental unit of $\Q(\sqrt{d})$ by $\eps = \frac{\delta}{2}\left( t+ u\sqrt{d}\right)$ where $\delta = 1$ if $d\equiv 1 \mod{4}$ and $\delta =2$ if $d\equiv 2,3\mod{4}$.  In $\C_p$ we have,  
    \[
    \frac{\log_p(\eps)}{\sqrt{d}} =\sum_{n=0}^\infty \frac{d^{n}}{2n+1}\left(\frac{u}{t} \right)^{2n+1}.
    \]
\end{lem}

\section{Power sums}
Throughout this section we assume that $\chi$ is a Dirichlet character with conductor $f$ and that $p$ is an odd prime coprime to $f$. As before, we set $F=\lcm(p,f) =pf$. We are interested in congruences for sums of the form:
\[
P(k, F, \chi) : = \frac{1}{F}\sum_{a=1}^F \chi(a) a^k\,\,\,\, \text{ and }\,\,\,\,\, P'(k, F, \chi) : =\frac{1}{F} \sum_{\substack{a=1\\ p\nmid a}}^F \chi(a) a^k. 
\]
These two quantities are related by the identity
\begin{equation}\label{eq: restricted sum}
    P'(k, F, \chi) = P(k, F, \chi) - \chi(p)p^{k-1} P(k, f, \chi).
\end{equation}
We recall that the power sums $P(k, F, \chi)$
are intimately related to generalized Bernoulli numbers via the formula \cite{IwasawaBook}:
\begin{equation}\label{eq: power sum}
P(k, F, \chi) = F^kB_{0, \chi} + \sum_{j=1}^k \binom{k}{j-1}\frac{B_{j,\chi}}{j}F^{k-j}.    
\end{equation}
Note that when $\chi$ is non-principal, $B_{0, \chi}=0$. When $\chi$ is principal, we will simply write $P(k,F)$ or $P'(k, F)$.  Thanks in large part to L. Carlitz \cite{Carlitz1959}, the arithmetic of generalized Bernoulli numbers is well understood.
\begin{thm}[Carlitz, 1959]\label{thm: Carlitz p-integral}
    Suppose $\chi$ is a  Dirichlet character of conductor $f>1$. If $f$ has at least two prime factors, then for any $n\geq 1$, $B_{n, \chi}/n$ is an algebraic integer. If $f=q^r$ for some prime $q$, then $B_{n, \chi}/n$ is $p$-integral for all primes $p$ coprime to $f$. 
\end{thm}

This is a generalization of the classical theorems of von Staudt--Clausen and Adams (see for example Theorem 2.7 and Theorem 2.11, respectively in \cite{MurtyPAdic}). Using equation (\ref{eq: power sum}) and Theorem \ref{thm: Carlitz p-integral} we have
\begin{lem}\label{lem: power sum mod p^2}
    Suppose $\chi$ is a non-principal Dirichlet character with conductor $f$ and let $k\geq 3$. Let $p$ be an odd prime coprime to $f$ and set $F=pf$.
    \begin{enumerate}[label=(\alph*)]
        \item If $\chi(-1) = (-1)^k$, then
        \[
        P(k, F, \chi) \equiv B_{k, \chi} \mod{p^2}
        \]
        unless $\chi(-1)=-1$ and $k=3$, in which case
        \[
       P(3, F, \chi) = B_{3, \chi} +FB_{1, \chi}.
        \]
        \item If $\chi(-1)\neq (-1)^k$, then 
        \[
       P(k, F, \chi) \equiv \frac{FB_{k-1, \chi}}{2} \mod{p^2}. 
        \]
    \end{enumerate}
    
\end{lem}

The analogous result for the principal character is similar but requires slightly more care. We prove the cases that we will need.  
\begin{lem}\label{lem: power sum for Bern}
    Let $p>3$ be a rational prime and set $F=p$.  Further, suppose that $3\leq k <p(p-1)$.  
    \begin{enumerate}[label=(\alph*)]
        \item If $(p-1)\mid k$, then 
        \[
            P(k, F) + \frac{1}{p} \equiv B_{k} + \frac{1}{p}\mod{p^2} 
        \]
        \item If $(-1)^k=1$ and $(p-1)\nmid k$, then 
        \[
        P(k, F) \equiv B_k + \frac{F^2k(k-1)}{6}B_{k-2} \mod{p^2}.
        \]
        
        \item If $(-1)^k =-1$, then 
        \[
        P(k, F) \equiv \frac{Fk}{2} B_{k-1} \mod{p^2}. 
        \]
    \end{enumerate}
\end{lem}
\begin{rem}
    We note that in (b), if $k\not\equiv 2 \mod{(p-1)}$, then $B_{k-2}$ is $p$-integral and hence $F^2B_{k-2}\equiv 0$ $\mod{p^2}$.  
\end{rem}
\begin{proof}
    We write
    \[
    P(k, F) = F^k -\frac{F^{k-1}}{2} + \sum_{j=2}^k \binom{k}{j-1} \frac{B_j}{j} F^{k-j}. 
    \]
    Adam's congruence ensures that $pB_j$ is $p$-integral for every $j$. We write the above sum as
    \[
    \sum_{j=2}^k \binom{k}{j-1} \frac{B_j}{j} F^{k-j} = B_k + \frac{Fk}{2}B_{k-1} + \frac{F^2k(k-1)}{6}B_{k-2} + \sum_{j=2}^{k-3} \binom{k}{j-1}\frac{B_j}{j} F^{k-j}. 
    \]
    By our assumption on $k$, $B_j/j$ can have at most a single power of $p$ in its denominator and hence for each $j$ such that $k-j\geq 3$, $F^{k-j}B_j/j$ is $p$-rational with numerator divisible by $p^2$. Therefore, 
    \begin{equation}\label{eq: bern power sum}
            P(k, F) = B_k + \frac{Fk}{2}B_{k-1} + \frac{F^2k(k-1)}{6} B_{k-2} + T
    \end{equation}
    where $T$ is $p$-rational and $T\equiv 0 \mod{p^2}$.  If $k\equiv 0 \mod{p-1}$, $B_{k-1}=0$ and $B_{k-2}$ is $p$-integral. Adding $1/p$ to both sides of equation (\ref{eq: bern power sum}) and reducing $\mod{p^2}$ we obtain (a). Parts (b) and (c) follow similarly. 
\end{proof}

\section{Congruence relations}\label{sec: congruences} 
In this section we will assume that $\chi$ is an even Dirichlet character with conductor $f$ and that $p$ is a rational prime greater than $3$. Keeping with prior notation, we set $F=\lcm(f, p)$.  
\newline 

The Kubota--Leopoldt $L$-function $L_p(s, \chi)$ can be written as \cite{washington1976}:
\begin{equation}\label{eq: washington expression}
    L_p(1-s, \chi) = -\frac{1}{s}\cdot \frac{1}{F} \sum_{\substack{a=1\\ p\nmid a}}^F \chi(a) \dia{a}^s \sum_{j=0}^\infty \binom{s}{j}\left(\frac{F}{a}\right)^j B_j.
\end{equation}
where
\[
\binom{x}{j} = \frac{x(x-1)\cdots (x-j+1)}{j!}.
\]
This is the usual binomial coefficient when $x$ is a positive integer. If the conductor of $\chi$ is not divisible by $p^2$, there is a power series representation of $L_p(1-s, \chi)$ of the form:
\[
L_p(1-s, \chi) = \sum_{i=0}^\infty a_is^i
\]
where $|a_0|\leq 1$ and $|a_i|<1$ for all $i\geq 1$. In fact, the argument in \cite[Theorem 5.12] {Washington} applies to the principal character as well, but in this case we have an expansion of the form
\[
\zeta_p(1-s) = -\frac{R}{s} + \sum_{i=0}^\infty a_is^i
\]
with $|a_0|_p\leq 1$ and $|a_i|_p<1$ \cite[Theorem 5.12]{Washington}. The aim of this section is to conduct a refined analysis of the coefficients $a_i$. We begin by rewriting the expression in (\ref{eq: washington expression}). The polynomial $\binom{x}{j}$ can be expressed as
\[
\binom{x}{j} = \frac{1}{j!}\sum_{k=0}^j \s(j,k) x^k, 
\]
where $\s(j,k)$ is the $(j,k)$-Stirling number of the first kind, i.e., $|\s(j,k)|$ is the number of permutations of $j$ elements consisting of $k$ disjoint cycles. 
Putting this formula into the power series for $L_p(1-s, \chi)$ and changing the order of summation we obtain:
\begin{equation}\label{eq: p-adic power series}
    L_p(1-s, \chi) = -\frac{1}{s}\cdot \frac{1}{F} \sum_{\substack{a=1\\ p\nmid a}}^F \chi(a) \dia{a}^s g_a(s)
\end{equation}
where 
\[
g_a(s) = \sum_{k=0}^\infty b_k(a)s^k
\]
and
\[
b_k(a) = \sum_{j=k}^\infty \left(\frac{F}{a} \right)^j \frac{B_j}{j!} \s(j,k).
\]

For the purposes of obtaining congruence relations for the values of $L_p(1-s, \chi)$, we need to determine the $p$-adic valuation of the $b_k(a)$. To this end, we note that the Bernoulli numbers $B_j$ are rational numbers with square-free denominators and that $v_p(j!) \leq \frac{j}{p-1}$. Therefore, we obtain a lower bound  
\[
v_p\left( \left(\frac{F}{a} \right)^j \frac{B_j}{j!}\right) \geq j\left(\frac{p-2}{p-1}\right)-1.
\]
From which it is clear that the series defining $b_k(a)$ converges. The worst case for this lower bound occurs when $p=5$. In particular, if $j\geq 6$ then 
\[
v_p\left( \left(\frac{F}{a} \right)^j \frac{B_j}{j!}\right) \geq 3.
\]

\begin{lem}\label{lem: g_a(s) appprox}
Suppose $s\in \Z_p$. If $p\geq 5$, then 
    \[
     g_a(s) \equiv b_0(a) + b_1(a)s + b_2(a)s^2 \mod{p^3}.
    \]
\end{lem}
\begin{proof}
By the above lower bound for the $p$-adic valuation for the terms in the sum defining $b_k(a)$, it suffices to determine the $p$-adic valuation for the coefficients with indices less than $6$. Omitting any terms divisible by at least $p^3$ we obtain: 
\begin{align*}
    b_0(a) &= 1\\
    b_1(a)&\equiv  -\frac{1}{2}\left(\frac{F}{a}\right)-\frac{1}{12}\left(\frac{F}{a}\right)^2 \mod{p^3}\\
    b_2(a) & \equiv \frac{1}{12}\left(\frac{F}{a}\right)^2 \mod{p^3}
\end{align*}
\end{proof}

Next we recall that 
\[
\dia{a}^s = \sum_{j=0}^\infty \frac{(\logp a)^j}{j!} s^j. 
\]
Moreover, $v_p((\logp a)^j/j!) \geq j(p-2)/(p-1)$. In particular, this is greater than $3$ when $j\geq 4$. Checking $j=0, 1, \ldots, 3$ we see that:
\[
\dia{a}^s \equiv 1  + \logp(a) s + \frac{(\logp a)^2}{2}s^2  \mod{p^3}.
\]
Moreover, for any integer $a$ coprime to $p$, we can write
\[
\logp(a) = \frac{1}{p-1} \logp(1+(a^{p-1}-1)) = \frac{1}{p-1} \sum_{j=1}^\infty (-1)^{j+1} \frac{(a^{p-1}-1)^j}{j}. 
\]
By Fermat's little theorem $a^{p-1}-1$ is divisible by $p$. Moreover, $v_p(j) \leq \frac{\log j}{\log p} \leq \frac{3j}{4}$ for all $j$ and primes $p\geq 5$. Writing $\F(a) = \frac{a^{p-1}-1}{p}$ we quickly deduce
\begin{lem}\label{lem: logp approx}
    Suppose $p>3$ is a rational prime and $a$ is an integer coprime to $p$. Then 
\[
\logp(a) \equiv \frac{1}{p-1} \left( p\F(a) - \frac{p^2\F(a)^2}{2}\right) \mod{p^3}. 
\]
\end{lem}
Using Lemma \ref{lem: logp approx} we obtain the following approximation for $\dia{a}^s$. 
\begin{lem}\label{lem: diamond approx}
    Suppose $s\in \Z_p$. If $p>3$, then 
    \[
    \dia{a}^s \equiv 1 + \frac{1}{p-1}\left( p\F(a) - \frac{p^2F(a)^2}{2}\right) s + \frac{1}{2} \left(\frac{pF(a)}{(p-1)}\right)^2 s^2 \mod{p^3}. 
    \]
\end{lem}

\begin{thm}\label{thm: power series}
    Let $\chi$ be a Dirichlet character with conductor $f$ and $p>3$ be a rational prime. Then for any $s\in \Z_p$ we have that
    \[
    L_p(1-s, \chi) \equiv \frac{a_{-1}}{s} + a_0 + a_1s
    \]
    where $a_{-1}=0$ if $\chi\neq \chi_0$,  $|a_0|_p\leq 1$, and $|a_1|_p <1$. 
\end{thm}
\begin{proof}
    This is now a matter of collecting the above results. Inserting Lemma \ref{lem: g_a(s) appprox} and Lemma \ref{lem: diamond approx} into (\ref{eq: p-adic power series}) we obtain
    \[
    L_p(1-s, \chi) = \frac{a_{-1}}{s} + a_0 + a_1s \mod{p^2}
    \]
    where
    \begin{align*}
        a_{-1} &= -\frac{1}{F} \sum_{\substack{a=1\\p\nmid a}}^F \chi(a) \\
        a_0 &\equiv -\frac{1}{F}\sum_{\substack{a=1\\p\nmid a}}^F \chi(a) \left(\frac{1}{p-1}\left( p\F(a) - \frac{p^2\F(a)^2}{2} \right) -\frac{1}{2}\left(\frac{F}{a}\right) - \frac{1}{12}\left(\frac{F}{a}\right)^2\right)\mod{p^2}\\
        a_1 &\equiv -\frac{1}{F}\sum_{\substack{a=1\\p\nmid a}}^F \chi(a) \left(\frac{p^2\F(a)^2}{2(p-1)^2} + \frac{F^2}{12a^2}  -\frac{p\F(a)}{2(p-1)}\left(\frac{F}{a}\right)  \right) \mod{p^2}.
    \end{align*}
   We write
    \[
    a_{-1} = -\frac{1}{F} \sum_{a=1}^F \chi(a) + \frac{1}{F}\sum_{b=1}^{F/p} \chi(pb).
    \]
    If $\chi$ is non-principal, the first sum is a complete character sum and hence zero. If $p\mid f$, then $\chi(pb)=0$ and hence the second sum is zero. If $p\nmid f$, then $f\mid F/p$ and again the second sum is a complete character sum and hence zero. When $\chi$ is trivial, we have that $a_{-1} = \frac{1}{p} -1$. We note that under the change of variables $s\mapsto 1-s$, we have moved the pole of $L_p(s, \chi_0)$ to $s=0$ and altered the sign of the residue. 
\end{proof}
When $\chi$ is principal, we define 
\[
\z(1-s) = L_p(1-s, \chi) +\frac{R}{s} 
\]
where $R= 1-p^{-1}$. This is an analytic function from $\Z_p\to \C_p$ with value $\z(1) = a_0$ as defined above. 
\begin{cor}\label{cor: congruence}
    Let $\chi$ be an even Dirichlet character with conductor $f$. Then for any integers $m$ and $n$, we have 
    \[
    L_p(1-m) \equiv L_p(1-n) \mod{p}
    \]
    and 
    \[
    L_p(1-m, \chi) \equiv L_p(1-n, \chi) +a_1(m-n) \mod{p^2}.
    \]
   If $\chi=\chi_0$ is principal, we replace $L_p(1-s, \chi_0)$ with $\z(1-s)$.  
\end{cor}
\begin{proof}
    By theorem \ref{thm: power series} we can write
    \[
    L_p(1-m, \chi) - a_1m \equiv a_0 \mod{p^2}
    \]
    for any integer $m$. Equating this congruence for different values of $m$ and $n$ yields second part of the corollary. To obtain the first, we note that as $|a_1|_p<1$,
    \[
    L_p(1-m, \chi) \equiv a_0 \mod{p}. 
    \]
    \end{proof}

\section{Evaluating $a_1\mod{p^2}$}
In this section we will determine $a_1\mod{p^2}$ for various quadratic characters. To simplify the exposition and calculations we will restrict to the cases of interest for our purposes. 

\subsection{Non-trivial quadratic characters}
  To fix notation, let $d=pm$ be a square-free positive integer and $p>3$ be a rational prime. Let $D=\delta^2 d$ and $\chi_D$ denote the discriminant and non-trivial character of $\Q(\sqrt{d})$ respectively. We define the character $\psim$ of conductor $\delta^2m$ by
    \[
    \chi_D(\cdot)  = \psim(\cdot) \LegSym{\cdot}{p}
    \]
    where $\LegSym{\cdot}{p}$ is the Legendre symbol $\mod{p}$. Finally, we set $r=\frac{p-1}{2}$, and note that $\psim(-1) = (-1)^r$. 

    Since $\chi_D$ has conductor $D$ and  $p\mid D$, the value $F$ introduced in the previous sections is simply equal to $D$. Moreover, recall that for $L_p(1-s, \chi_D)$ that 
    \[
   a_1 \equiv -\frac{1}{F}\sum_{\substack{a=1\\p\nmid a}}^F \chi_D(a) \left(\frac{p^2\F(a)^2}{2(p-1)^2} + \frac{F^2}{12a^2}  -\frac{p\F(a)}{2(p-1)}\left(\frac{F}{a}\right)  \right) \mod{p^2}.
    \]

\begin{lem}\label{lem: a_1 non-princ}
    Let $d=pm$, $D$, and $r$ be as above. Suppose $d>5$.  Then 
    \[
     a_1 \equiv -\frac{1}{2r^2}\left(\frac{B_{3r, \psim}}{3} -(1-\psim(p)p^{r-1})B_r \right)\mod{p^2}.
    \]
\end{lem}
\begin{rem}
    When $p>5$, $(1-\psim(p)p^{r-1}) \equiv 1 \mod{p^2}$.  
\end{rem}
\begin{proof}
    Set $a_1 \equiv S_1 + S_2 +S_3 \mod{p^2}$. We handle $S_1$ first. Using the definition of $\F(a)$ we have that
    \[
    S_1 \equiv -\frac{1}{2(p-1)^2} \cdot \frac{1}{F} \sum_{\substack{a=1\\p\nmid a}}^F \chi_D(a)(a^{2(p-1)} - 2a^{p-1} +1) \mod{p^2}.
    \]
    First we note that
    \[
    \sum_{\substack{a=1\\p\nmid a}}^F \chi_D(a) = 0. 
    \]
    Secondly, we note that due to the factor of $1/F$ in front of the sum, we have that $\frac{1}{F}\LegSym{a}{p} \equiv a^{p^2r}\mod{p^2}$ for all $a$ coprime to $p$. Hence,  
    \[
    S_1 \equiv -\frac{1}{2(p-1)^2} \cdot \frac{1}{F} \sum_{\substack{a=1\\p\nmid a}}^F \psim(a)\left(a^{(p^2+4)r} - 2a^{(p^2+2)r}\right) \mod{p^2}. 
    \]
    Since $\delta^2m$ is coprime to $p$, Lemma \ref{lem: power sum mod p^2} yields
    \begin{equation}\label{eq: S_1 non-trivial char}
         S_1 \equiv -\frac{1}{2(p-1)^2} \left( B_{(p^2+4)r, \psim} -2B_{(p^2+2)r, \psim} \right) \mod{p^2}.
    \end{equation}
    We can simplify this expression using Z.-H. Sun's Theorem 8.1(a) of \cite{Sun2000}. This asserts that if $\chi$ is a Dirichlet character, then for all positive integers $b$ that are not divisible by $p-1$ and positive integers $k$, that  
    \[
    \frac{B_{k(p-1)+b, \chi}}{k(p-1)+b} \equiv k\frac{B_{p-1+b, \chi}}{p-1+b} - (k-1)(1-\chi(p)p^{b-1})\frac{B_{b, \chi}}{b} \mod{p^2}.
    \]
    Observing that $(p^2+4)r = \left( \frac{p^2+3}{2}\right)(p-1) + r$ and $(p^2+2)r = \left(\frac{p^2+1}{2}\right)(p-1) + r$, we apply Sun's congruence to the identity in (\ref{eq: S_1 non-trivial char}) to  deduce that
    \[
    S_1 \equiv -\frac{1}{2r^2}\left(\frac{B_{3r, \psim}}{3} -(1-\psim(p)p^{r-1})B_r \right) \mod{p^2}. 
    \]
    
    When $d>5$, showing that $S_2 \equiv S_3 \equiv 0 \mod{p^2}$ follows the same recipe as above using Lemma \ref{lem: power sum mod p^2} in conjunction with Theorem \ref{thm: Carlitz p-integral}. \newline 

\end{proof}
\begin{rem}
        When $d=5$, one has to appeal to Lemma \ref{lem: power sum for Bern} to keep track of the terms that arise from $S_2$ and $S_3$. In particular, it is a straightforward calculation using the methods of the proceeding proof to show that  
    \begin{align*}
        S_2 &\equiv -\frac{F^2}{12}B_{48} \mod{5^2}\\
        S_3 &\equiv \frac{F^2}{8}\left(3B_{52}+B_{48} \right) \mod{5^2}.
    \end{align*}
\end{rem}
\subsection{The principal character}
Let $F=p$ and $R= 1- p^{-1}$. For the modified $p$-adic zeta function $\zeta_p^*(1-s)$, we have that 
 \begin{align*}
     a_0 &\equiv -\frac{1}{p}\sum_{\substack{a=1\\p\nmid a}}^p \left( \logp(a) -\frac{1}{2}\left(\frac{p}{a}\right) - \frac{1}{12}\left(\frac{p}{a}\right)^2\right)\mod{p^2}\\
        a_1 &\equiv -\frac{1}{p}\sum_{\substack{a=1\\p\nmid a}}^p  \left(\frac{p^2\F(a)^2}{2(p-1)^2} + \frac{p^2}{12a^2}  -\frac{p\F(a)}{2(p-1)}\left(\frac{p}{a}\right)  \right) \mod{p^2}.
 \end{align*}
 \begin{lem} \label{lem: a_0 and a_1 princ}
     Suppose $p>3$ and let $W_p = \frac{(p-1)!+1}{p}$. 
     \begin{enumerate}[label=(\alph*)]
         \item $a_0 \equiv W_p\left(1 + \frac{pW_p}{2} \right) \mod{p^2}$. 
         \item $a_0 \equiv W_p \mod{p}$.
         \item $a_1 \equiv -\frac{1}{2(p-1)^2}(B_{2(p-1)} - 2B_{p-1} +R) \mod{p^2}$. 
     \end{enumerate}
 \end{lem}
\begin{proof}
   Since $pW_p$ is divisible by $p$, it is plain that (b) follows from (a). To prove (a) we observe that $a_0$ can be written as
   \[
   a_0 \equiv -\frac{1}{p} \left(\logp((p-1)!) -\frac{p}{2}\sum_{a=1}^{p-1}a^{-1} - \frac{p^2}{12}\sum_{a=1}^{p-1} a^{-2} \right) \mod{p^2}.
   \]
   Noting that $a^{-1} \equiv a^{\phi(p^2)-1}\mod{p^2}$ and $a^{-2} \equiv a^{\phi(p^2)-2} \mod{p^2}$ we deduce from Lemma \ref{lem: power sum for Bern} that
   \[
   \sum_{a=1}^{p-1}a^{\phi(p^2)-1}  \equiv  p\sum_{a=1}^{p-1} a^{\phi(p^2)-2} \equiv 0 \mod{p^2}. 
   \]
   On the other hand, we can write $(p-1)! = -1 + pW_p$ so that 
   \[
   \frac{\logp((p-1)!)}{p} = \sum_{k=1}^\infty \frac{p^{k-1}W_p^k}{k}.
   \]
  Since $pW_p$ is divisible by $p$, we deduce that  $\frac{\logp((p-1)!)}{p} \equiv W_p + \frac{pW_p^2}{2} \mod{p^2}$. \newline

   Part (c) follows in essentially the same way as the previous lemma but requires slightly more care. By the same argument as above it is straightforward to show that 
   \[
   a_1 \equiv -\frac{1}{2(p-1)^2} \cdot \frac{1}{p} \sum_{a=1}^{p-1} \left(a^{2(p-1)} -2a^{p-1} +1 \right) \mod{p^2}. 
   \]
   We next observe that $\frac{1}{p}\sum_{a=1}^{p-1}1 = 1-p^{-1}$. Adding and subtracting $p^{-1}$ to the above expression and applying Lemma \ref{lem: power sum for Bern} we have
   \[
   a_1 \equiv -\frac{1}{2(p-1)^2}\left( B_{2(p-1)} + \frac{1}{p} - 2\left(B_{p-1} + \frac{1}{p}\right) + 1\right) \mod{p^2}. 
   \]
   Collecting the $1/p$ terms and recalling that $R= 1 - p^{-1}$ we conclude the result. 
\end{proof}

\section{Proof of Theorems}
As above, we let $r= \frac{p-1}{2}$. We additionally note that the $r$-th power of the Teichm\"uller character is the Legendre symbol $\mod{p}$,  so that $\chi_D\w^r = \psim$  (see Section 2.2 of \cite{Fellini2025} for a detailed discussion). 

\begin{proof}[Proof of Theorem \ref{thm: thm 1}]
Let $r = \frac{p-1}{2}$. Then by the second part of Corollary \ref{cor: congruence} we have that 
  \[
  L_p(1, \chi_D) \equiv L_p(1 - r, \chi_D) - ra_1 \mod{p^2}. 
  \]
Up to this point, we have ignored the issue of embeddings. Here, we choose the embedding of $\overline{\Q}$ into $\C_p$ such that the $p$-adic class number formula is an identity in $\C_p$. Then, by the $p$-adic class number formula (see Theorem 5.24 in \cite{Washington}) and Proposition \ref{prop: p-adic log of fund. unit} we have that
  \[
  L_p(1, \chi_D)  \equiv \frac{2h}{\delta} \left( \frac{u}{t} + \frac{d}{3}\left( \frac{u}{t} \right)^3\right) \mod{p^2}. 
  \]
  On the other hand, 
  \[
  L_p(1-r, \chi_D) \equiv -(1- \psim(p)p^{r-1})\frac{B_{r, \psim}}{r} \mod{p^2}.
  \]
  Inserting the expression from Lemma \ref{lem: a_1 non-princ} we conclude that
  \[
  \frac{2h}{\delta} \left( \frac{u}{t} + \frac{d}{3}\left( \frac{u}{t} \right)^3\right) \equiv -3(1-\psim(p)p^{r-1})\frac{B_{r, \psim}}{2r} +\frac{B_{3r, \psim}}{6r} \mod{p^2}. 
  \]
\end{proof}

\begin{proof}[Proof of Theorem \ref{thm: thm 2}]
  Suppose $p>3$. Then by Corollary \ref{cor: congruence} we have that
    \[
    \zeta_p^*(1) \equiv \zeta_p^*(1-k(p-1)) \mod{p}. 
    \]
    By Lemma \ref{lem: a_0 and a_1 princ}, $\zeta_p^*(1) = a_0 \equiv W_p\mod{p}$. On the other hand,
    \[
    \zeta_p^*(1-k(p-1)) \equiv -\frac{B_{k(p-1)}}{k(p-1)} + \frac{R}{k(p-1)}\mod{p}.
    \]
    Multiplying through by $k$ we have
    \[
    kW_p \equiv B_{k(p-1)} +\frac{1}{p} - 1 \mod{p}. 
    \]
\end{proof}

\begin{proof}[Proof of Theorem \ref{thm: thm 3}]
  By Lemma \ref{lem: a_0 and a_1 princ} we have that \[
  a_0 \equiv W_p\left(1+ \frac{pW_p}{2}\right) \mod{p^2}. 
  \]  
  Similarly, $\zeta_p^*(1-k(p-1)) \equiv -\frac{B_{k(p-1)}}{k(p-1)} + \frac{R}{k(p-1)} \mod{p^2}$. By Corollary \ref{cor: congruence} we have that 
    \[
    k(p-1)W_p\left(1 + \frac{pW_p}{2} \right) \equiv -B_{k(p-1)} +R  -k(B_{p-1} - B_{2(p-1)}) - \frac{k}{2}(B_{2(p-1)}-R) \mod{p^2}.\]
\end{proof}
\section*{Acknowledgments} 
The research of the author was partially supported by an Ontario Graduate Scholarship. The author would like to thank M. Ram Murty for his support and helpful comments on earlier versions of this work. Additionally, the author wishes to express gratitude to Curtis Grant and Yuveshen Mooroogen for their helpful comments, which greatly improved the  exposition.
\printbibliography

@article {Agoh2016,
    AUTHOR = {Agoh, Takashi},
     TITLE = {Congruences related to the {A}nkeny-{A}rtin-{C}howla
              conjecture},
   JOURNAL = {Integers},
  FJOURNAL = {Integers. Electronic Journal of Combinatorial Number Theory},
    VOLUME = {16},
      YEAR = {2016},
     PAGES = {Paper No. A12, 30},
      ISSN = {1553-1732},
   MRCLASS = {11R29 (11B68)},
  MRNUMBER = {3475253},
MRREVIEWER = {Su\ Hu},
}

@article {aktas,
    AUTHOR = {Akta{\c s}, Kevser and Ram Murty, M.},
     TITLE = {Fundamental units and consecutive squarefull numbers},
   JOURNAL = {Int. J. Number Theory},
  FJOURNAL = {International Journal of Number Theory},
    VOLUME = {13},
      YEAR = {2017},
    NUMBER = {1},
     PAGES = {243--252},
      ISSN = {1793-0421,1793-7310},
   MRCLASS = {11A25 (11D09)},
  MRNUMBER = {3573422},
MRREVIEWER = {Shichun\ Yang},
       DOI = {10.1142/S1793042117500142},
       URL = {https://doi.org/10.1142/S1793042117500142},
}

@article {Ankeny1951,
    AUTHOR = {Ankeny, N. C. and Artin, E. and Chowla, S.},
     TITLE = {The class-number of real quadratic fields},
   JOURNAL = {Proc. Nat. Acad. Sci. U.S.A.},
  FJOURNAL = {Proceedings of the National Academy of Sciences of the United
              States of America},
    VOLUME = {37},
      YEAR = {1951},
     PAGES = {524--525},
      ISSN = {0027-8424},
   MRCLASS = {10.0X},
  MRNUMBER = {43137},
MRREVIEWER = {W.\ H.\ Mills},
       DOI = {10.1073/pnas.37.8.524},
       URL = {https://doi.org/10.1073/pnas.37.8.524},
}

@article {Ankeny1952,
    AUTHOR = {Ankeny, N. C. and Artin, E. and Chowla, S.},
     TITLE = {The class-number of real quadratic number fields},
   JOURNAL = {Ann. of Math. (2)},
  FJOURNAL = {Annals of Mathematics. Second Series},
    VOLUME = {56},
      YEAR = {1952},
     PAGES = {479--493},
      ISSN = {0003-486X},
   MRCLASS = {10.0X},
  MRNUMBER = {49948},
MRREVIEWER = {G.\ Hochschild},
       DOI = {10.2307/1969656},
       URL = {https://doi.org/10.2307/1969656},
}

@article {AC,
    AUTHOR = {Ankeny, N. C. and Chowla, S.},
     TITLE = {A note on the class number of real quadratic fields},
   JOURNAL = {Acta Arith.},
  FJOURNAL = {Polska Akademia Nauk. Instytut Matematyczny. Acta Arithmetica},
    VOLUME = {6},
      YEAR = {1960},
     PAGES = {145--147},
      ISSN = {0065-1036},
   MRCLASS = {10.00},
  MRNUMBER = {115983},
MRREVIEWER = {L.\ Carlitz},
       DOI = {10.4064/aa-6-2-145-147},
       URL = {https://doi.org/10.4064/aa-6-2-145-147},
}

@article {Beeger1913,
    AUTHOR = {Beeger, N. G. W. H.},
     TITLE = {Quelques remarques sur les congruences {$r^{p-1}\equiv 1 \mod{p^2}$} et {$(p-1)!\equiv -1 \mod{p^2}$}},
   JOURNAL = {Messenger Math.},
  FJOURNAL = {The Messenger of Mathematics},
    VOLUME = {43},
      YEAR = {1913},
     PAGES = {72--84}
}

@article {Benmerieme2024,
    AUTHOR = {Benmerieme, Y. and Movahhedi, A.},
     TITLE = {Ankeny-{A}rtin-{C}howla and {M}ordell conjectures in terms of
              {$p$}-rationality},
   JOURNAL = {J. Number Theory},
  FJOURNAL = {Journal of Number Theory},
    VOLUME = {257},
      YEAR = {2024},
     PAGES = {202--214},
      ISSN = {0022-314X,1096-1658},
   MRCLASS = {11R42 (11R11 11R32)},
  MRNUMBER = {4673311},
       DOI = {10.1016/j.jnt.2023.10.013},
       URL = {https://doi.org/10.1016/j.jnt.2023.10.013},
}

@misc{Boudot2022,
author = {Boudot, Fabrice and Gaudry, Pierrick and Guillevic, Aurore and Heninger, Nadia and Thom{\'e}, Emmanuel and Zimmermann, Paul},
address = {New York},
copyright = {Copyright The Institute of Electrical and Electronics Engineers, Inc. (IEEE) 2022},
issn = {1540-7993},
journal = {IEEE security & privacy},
language = {eng},
number = {2},
pages = {80-86},
publisher = {IEEE},
title = {The State of the Art in Integer Factoring and Breaking Public-Key Cryptography},
volume = {20},
year = {2022},
}

@article {Carlitz1953a,
    AUTHOR = {Carlitz, L.},
     TITLE = {Note on the class number of real quadratic fields},
   JOURNAL = {Proc. Amer. Math. Soc.},
  FJOURNAL = {Proceedings of the American Mathematical Society},
    VOLUME = {4},
      YEAR = {1953},
     PAGES = {535--537},
      ISSN = {0002-9939,1088-6826},
   MRCLASS = {10.0X},
  MRNUMBER = {56636},
MRREVIEWER = {A.\ L.\ Whiteman},
       DOI = {10.2307/2032518},
       URL = {https://doi.org/10.2307/2032518},
}

@article {Carlitz1953b,
    AUTHOR = {Carlitz, L.},
     TITLE = {The class number of an imaginary quadratic field},
   JOURNAL = {Comment. Math. Helv.},
  FJOURNAL = {Commentarii Mathematici Helvetici},
    VOLUME = {27},
      YEAR = {1953},
     PAGES = {338--345},
      ISSN = {0010-2571,1420-8946},
   MRCLASS = {10.0X},
  MRNUMBER = {58648},
MRREVIEWER = {A.\ L.\ Whiteman},
       DOI = {10.1007/BF02564567},
       URL = {https://doi.org/10.1007/BF02564567},
}

@article {Carlitz1959,
    AUTHOR = {Carlitz, L.},
     TITLE = {Arithmetic properties of generalized {B}ernoulli numbers},
   JOURNAL = {J. Reine Angew. Math.},
  FJOURNAL = {Journal f\"ur die Reine und Angewandte Mathematik. [Crelle's
              Journal]},
    VOLUME = {202},
      YEAR = {1959},
     PAGES = {174--182},
      ISSN = {0075-4102,1435-5345},
   MRCLASS = {10.00},
  MRNUMBER = {109132},
MRREVIEWER = {A.\ L.\ Whiteman},
       DOI = {10.1515/crll.1959.202.174},
       URL = {https://doi.org/10.1515/crll.1959.202.174},
}

@book {Cohen1993,
    AUTHOR = {Cohen, Henri},
     TITLE = {A course in computational algebraic number theory},
    SERIES = {Graduate Texts in Mathematics},
    VOLUME = {138},
 PUBLISHER = {Springer-Verlag, Berlin},
      YEAR = {1993},
     PAGES = {xii+534},
      ISBN = {3-540-55640-0},
   MRCLASS = {11Y40 (11Rxx 68Q40)},
  MRNUMBER = {1228206},
MRREVIEWER = {Joe\ P.\ Buhler},
       DOI = {10.1007/978-3-662-02945-9},
       URL = {https://doi.org/10.1007/978-3-662-02945-9},
}

@article {Cohen2020,
    AUTHOR = {Cohen, Henri and Thorne, Frank},
     TITLE = {On {$D_\ell$}-extensions of odd prime degree {$\ell$}},
   JOURNAL = {Proc. Lond. Math. Soc. (3)},
  FJOURNAL = {Proceedings of the London Mathematical Society. Third Series},
    VOLUME = {121},
      YEAR = {2020},
    NUMBER = {5},
     PAGES = {1171--1206},
      ISSN = {0024-6115,1460-244X},
   MRCLASS = {11R20 (11N45 11R29 11R45)},
  MRNUMBER = {4133705},
MRREVIEWER = {St\'ephane\ R.\ Louboutin},
       DOI = {10.1112/plms.12339},
       URL = {https://doi.org/10.1112/plms.12339},
}

@article {Ellenberg2007,
    AUTHOR = {Ellenberg, Jordan S. and Venkatesh, Akshay},
     TITLE = {Reflection principles and bounds for class group torsion},
   JOURNAL = {Int. Math. Res. Not. IMRN},
  FJOURNAL = {International Mathematics Research Notices. IMRN},
      YEAR = {2007},
    NUMBER = {1},
     PAGES = {Art. ID rnm002, 18},
      ISSN = {1073-7928,1687-0247},
   MRCLASS = {11R29 (11R47)},
  MRNUMBER = {2331900},
MRREVIEWER = {\'Alvaro\ Lozano-Robledo},
       DOI = {10.1093/imrn/rnm002},
       URL = {https://doi.org/10.1093/imrn/rnm002},
}

@incollection{Fellini2024,
    AUTHOR = {{Fellini}, N. and Ram Murty, M. },
     TITLE = {Fermat quotients and the {A}nkeny-{A}rtin-{C}howla conjecture},
 BOOKTITLE = {Class groups of number fields and related topics},
    SERIES = {Springer Proc. Math. Stat.},
    VOLUME = {470},
     PAGES = {1--23},
 PUBLISHER = {Springer, Singapore},
      YEAR = {2024},
      ISBN = {978-981-97-6910-0},
   MRCLASS = {11Rxx},
  MRNUMBER = {4866961},
       DOI = {10.1007/978-981-97-6911-7\_1},
       URL = {https://doi.org/10.1007/978-981-97-6911-7_1},
}

@article{Fellini2025,
author = {{Fellini}, N.},
title = {Congruence relations of {A}nkeny–{A}rtin–{C}howla type for real quadratic fields},
journal ={Int. J. Number Theory},
fjournal = {International Journal of Number Theory},
volume = {21},
number = {06},
pages = {1317-1335},
year = {2025},
doi = {10.1142/S179304212550068X},

URL = { https://doi.org/10.1142/S179304212550068X
      }
}

@book {IwasawaBook,
    AUTHOR = {Iwasawa, Kenkichi},
     TITLE = {Lectures on {$p$}-adic {$L$}-functions},
    SERIES = {Annals of Mathematics Studies},
    VOLUME = {No. 74},
 PUBLISHER = {Princeton University Press, Princeton, NJ; University of Tokyo
              Press, Tokyo},
      YEAR = {1972},
     PAGES = {vii+106},
   MRCLASS = {12A70 (12B30)},
  MRNUMBER = {360526},
MRREVIEWER = {Yasumasa\ Akagawa},
}

@article {Iwasawa1969,
    AUTHOR = {Iwasawa, Kenkichi},
     TITLE = {On {$p$}-adic {$L$}-functions},
   JOURNAL = {Ann. of Math. (2)},
  FJOURNAL = {Annals of Mathematics. Second Series},
    VOLUME = {89},
      YEAR = {1969},
     PAGES = {198--205},
      ISSN = {0003-486X},
   MRCLASS = {10.67},
  MRNUMBER = {269627},
MRREVIEWER = {R.\ Ayoub},
       DOI = {10.2307/1970817},
       URL = {https://doi.org/10.2307/1970817},
}

@article {Jakubec1996,
    AUTHOR = {Jakubec, Stanislav},
     TITLE = {Congruence of {A}nkeny-{A}rtin-{C}howla type modulo {$p^2$}
              for cyclic fields of prime degree {$l$}},
   JOURNAL = {Acta Arith.},
  FJOURNAL = {Acta Arithmetica},
    VOLUME = {74},
      YEAR = {1996},
    NUMBER = {4},
     PAGES = {293--310},
      ISSN = {0065-1036,1730-6264},
   MRCLASS = {11R29 (11R20)},
  MRNUMBER = {1378225},
       DOI = {10.4064/aa-74-4-293-310},
       URL = {https://doi.org/10.4064/aa-74-4-293-310},
}

@article {Jakubec1998a,
    AUTHOR = {Jakubec, Stanislav},
     TITLE = {Note on the congruence of {A}nkeny-{A}rtin-{C}howla type
              modulo {$p^2$}},
   JOURNAL = {Acta Arith.},
  FJOURNAL = {Acta Arithmetica},
    VOLUME = {85},
      YEAR = {1998},
    NUMBER = {4},
     PAGES = {377--388},
      ISSN = {0065-1036,1730-6264},
   MRCLASS = {11R29 (11R18)},
  MRNUMBER = {1641062},
MRREVIEWER = {Peter\ Stevenhagen},
       DOI = {10.4064/aa-85-4-377-388},
       URL = {https://doi.org/10.4064/aa-85-4-377-388},
}

@article {Jakubec1998b,
    AUTHOR = {Jakubec, Stanislav and La\v s\v s\'ak, Miroslav},
     TITLE = {Congruence of {A}nkeny-{A}rtin-{C}howla type modulo {$p^2$}},
   JOURNAL = {Ann. Math. Sil.},
  FJOURNAL = {Annales Mathematicae Silesianae},
    NUMBER = {12},
      YEAR = {1998},
     PAGES = {75--91},
      ISSN = {0860-2107,2391-4238},
   MRCLASS = {11R29 (11R18)},
  MRNUMBER = {1673072},
MRREVIEWER = {Peter\ Stevenhagen},
}

@article {Jakubec2013,
    AUTHOR = {Jakubec, Stanislav and Marko, Franti\v sek},
     TITLE = {Ankeny-{A}rtin-{C}howla type congruences modulo {$p^3$}},
   JOURNAL = {Math. Slovaca},
  FJOURNAL = {Mathematica Slovaca},
    VOLUME = {63},
      YEAR = {2013},
    NUMBER = {6},
     PAGES = {1183--1208},
      ISSN = {0139-9918,1337-2211},
   MRCLASS = {11R29 (11R18 11R27)},
  MRNUMBER = {3146614},
MRREVIEWER = {Akito\ Nomura},
       DOI = {10.2478/s12175-013-0165-7},
       URL = {https://doi.org/10.2478/s12175-013-0165-7},
}

@book{Koblitz,
  title={p-adic Numbers, p-adic Analysis, and Zeta-Functions},
  author={Koblitz, N.},
  isbn={9780387960173},
  lccn={96152654},
  series={Graduate Texts in Mathematics},
  url={https://books.google.ca/books?id=13qWQsBY-NgC},
  year={1996},
  publisher={Springer New York}
}

@article {Kubota1964,
    AUTHOR = {Kubota, Tomio and Leopoldt, Heinrich-Wolfgang},
     TITLE = {Eine {$p$}-adische {T}heorie der {Z}etawerte. {I}.
              {E}inf\"uhrung der {$p$}-adischen {D}irichletschen
              {$L$}-{F}unktionen},
   JOURNAL = {J. Reine Angew. Math.},
  FJOURNAL = {Journal f\"ur die Reine und Angewandte Mathematik. [Crelle's
              Journal]},
    VOLUME = {214/215},
      YEAR = {1964},
     PAGES = {328--339},
      ISSN = {0075-4102,1435-5345},
   MRCLASS = {10.67 (10.41)},
  MRNUMBER = {163900},
MRREVIEWER = {D.\ J.\ Lewis},
       DOI = {10.1515/crll.1964.214-215.328},
       URL = {https://doi.org/10.1515/crll.1964.214-215.328},
}

@article {Le1994,
    AUTHOR = {Le, Mao Hua},
     TITLE = {Upper bounds for class numbers of real quadratic fields},
   JOURNAL = {Acta Arith.},
  FJOURNAL = {Acta Arithmetica},
    VOLUME = {68},
      YEAR = {1994},
    NUMBER = {2},
     PAGES = {141--144},
      ISSN = {0065-1036,1730-6264},
   MRCLASS = {11R29 (11M20 11R11)},
  MRNUMBER = {1305196},
MRREVIEWER = {St\'{e}phane\ R.\ Louboutin},
       DOI = {10.4064/aa-68-2-141-144},
       URL = {https://doi.org/10.4064/aa-68-2-141-144},
}

@article {Lehmer1938,
    AUTHOR = {Lehmer, Emma},
     TITLE = {On congruences involving {B}ernoulli numbers and the quotients
              of {F}ermat and {W}ilson},
   JOURNAL = {Ann. of Math. (2)},
  FJOURNAL = {Annals of Mathematics. Second Series},
    VOLUME = {39},
      YEAR = {1938},
    NUMBER = {2},
     PAGES = {350--360},
      ISSN = {0003-486X,1939-8980},
   MRCLASS = {99-04},
  MRNUMBER = {1503412},
       DOI = {10.2307/1968791},
       URL = {https://doi.org/10.2307/1968791},
}

@article {Lemmermeyer2005,
    AUTHOR = {Lemmermeyer, Franz},
     TITLE = {Class groups of dihedral extensions},
   JOURNAL = {Math. Nachr.},
  FJOURNAL = {Mathematische Nachrichten},
    VOLUME = {278},
      YEAR = {2005},
    NUMBER = {6},
     PAGES = {679--691},
      ISSN = {0025-584X,1522-2616},
   MRCLASS = {11R29 (11R20)},
  MRNUMBER = {2135500},
MRREVIEWER = {Michael\ R.\ Bush},
       DOI = {10.1002/mana.200310263},
       URL = {https://doi.org/10.1002/mana.200310263},
}

@article {Louboutin2002,
    AUTHOR = {Louboutin, St\'ephane},
     TITLE = {Computation of class numbers of quadratic number fields},
   JOURNAL = {Math. Comp.},
  FJOURNAL = {Mathematics of Computation},
    VOLUME = {71},
      YEAR = {2002},
    NUMBER = {240},
     PAGES = {1735--1743},
      ISSN = {0025-5718,1088-6842},
   MRCLASS = {11R29 (11R11)},
  MRNUMBER = {1933052},
MRREVIEWER = {Yoonjin\ Lee},
       DOI = {10.1090/S0025-5718-01-01367-9},
       URL = {https://doi.org/10.1090/S0025-5718-01-01367-9},
}

@article {Mollin1986,
    AUTHOR = {Mollin, R. A. and Walsh, P. G.},
     TITLE = {A note on powerful numbers, quadratic fields and the
              {P}ellian},
   JOURNAL = {C. R. Math. Rep. Acad. Sci. Canada},
  FJOURNAL = {La Soci\'et\'e{} Royale du Canada. L'Academie des Sciences.
              Comptes Rendus Math\'ematiques. (Mathematical Reports)},
    VOLUME = {8},
      YEAR = {1986},
    NUMBER = {2},
     PAGES = {109--114},
      ISSN = {0706-1994},
   MRCLASS = {11A51},
  MRNUMBER = {831787},
MRREVIEWER = {Roger\ Paysant-Le Roux},
}

@article {Mordell1961,
    AUTHOR = {Mordell, L. J.},
     TITLE = {On a {P}ellian equation conjecture. {II}},
   JOURNAL = {J. London Math. Soc.},
  FJOURNAL = {The Journal of the London Mathematical Society},
    VOLUME = {36},
      YEAR = {1961},
     PAGES = {282--288},
      ISSN = {0024-6107}
}

@book {MurtyPAdic,
    AUTHOR = {Ram Murty, M. },
     TITLE = {Introduction to {$p$}-adic analytic number theory},
    SERIES = {AMS/IP Studies in Advanced Mathematics},
    VOLUME = {27},
 PUBLISHER = {American Mathematical Society, Providence, RI; International
              Press, Somerville, MA},
      YEAR = {2002},
     PAGES = {x+149},
      ISBN = {0-8218-3262-X},
   MRCLASS = {11S80 (11-01 11S40)},
  MRNUMBER = {1913413},
MRREVIEWER = {Vinayak\ Vatsal},
       DOI = {10.1090/amsip/027},
       URL = {https://doi.org/10.1090/amsip/027},
}

@book {MurtyProbsAlg,
    AUTHOR = {Ram Murty, M.  and Esmonde, Jody},
     TITLE = {Problems in algebraic number theory},
    SERIES = {Graduate Texts in Mathematics},
    VOLUME = {190},
   EDITION = {Second},
 PUBLISHER = {Springer-Verlag, New York},
      YEAR = {2005},
     PAGES = {xvi+352},
      ISBN = {0-387-22182-4},
   MRCLASS = {11Rxx (11-01)},
  MRNUMBER = {2090972},
}

@article {Reinhart2024a,
    AUTHOR = {{Reinhart}, {Andreas}},
     TITLE = {A counterexample to the {P}ellian equation conjecture of {M}ordell},
   JOURNAL = {Acta Arith.},
  FJOURNAL = {Acta Arithmetica},
    VOLUME = {215},
      YEAR = {2024},
    NUMBER = {1},
     PAGES = {85--95},
      ISSN = {0065-1036,1730-6264},
   MRCLASS = {11R11 (11R27)},
  MRNUMBER = {4772273},
       DOI = {10.4064/aa240214-3-4},
       URL = {https://doi.org/10.4064/aa240214-3-4},
}

@misc{Reinhart2024b,
	title = {A counterexample to the {Conjecture} of {Ankeny}, {Artin} and {Chowla}},
	url = {http://arxiv.org/abs/2410.21864},
	language = {en},
	urldate = {2024-11-07},
	publisher = {arXiv},
	author = {{Reinhart}, {Andreas}},
	month = {10},
	year = {2024},
	note = {arXiv:2410.21864 [math]},
}

@article {Scholz1932,
    AUTHOR = {Scholz, Arnold},
     TITLE = {\"Uber die {B}eziehung der {K}lassenzahlen quadratischer
              {K}\"orper zueinander},
   JOURNAL = {J. Reine Angew. Math.},
  FJOURNAL = {Journal f\"ur die Reine und Angewandte Mathematik. [Crelle's
              Journal]},
    VOLUME = {166},
      YEAR = {1932},
     PAGES = {201--203},
      ISSN = {0075-4102,1435-5345},
   MRCLASS = {99-04},
  MRNUMBER = {1581309},
       DOI = {10.1515/crll.1932.166.201},
       URL = {https://doi.org/10.1515/crll.1932.166.201},
}

@article {Sheingorn1989,
    AUTHOR = {Sheingorn, Mark},
     TITLE = {Hyperbolic reflections on {P}ell's equation},
   JOURNAL = {J. Number Theory},
  FJOURNAL = {Journal of Number Theory},
    VOLUME = {33},
      YEAR = {1989},
    NUMBER = {3},
     PAGES = {267--285},
      ISSN = {0022-314X,1096-1658},
   MRCLASS = {20H05 (11D09 11F06)},
  MRNUMBER = {1027055},
MRREVIEWER = {Troels\ J\o rgensen},
       DOI = {10.1016/0022-314X(89)90064-4},
       URL = {https://doi.org/10.1016/0022-314X(89)90064-4},
}

@article {Stephens1988,
    AUTHOR = {Stephens, A. J. and Williams, H. C.},
     TITLE = {Some computational results on a problem concerning powerful numbers},
   JOURNAL = {Math. Comp.},
  FJOURNAL = {Mathematics of Computation},
    VOLUME = {50},
      YEAR = {1988},
    NUMBER = {182},
     PAGES = {619--632},
      ISSN = {0025-5718,1088-6842},
   MRCLASS = {11R11 (11A51 11R27 11Y16 11Y40)},
  MRNUMBER = {929558},
MRREVIEWER = {H.\ J.\ Godwin},
       DOI = {10.2307/2008629},
       URL = {https://doi.org/10.2307/2008629},
}

@incollection {Stevenhagen2008,
    AUTHOR = {Stevenhagen, Peter},
     TITLE = {The number field sieve},
 BOOKTITLE = {Algorithmic number theory: lattices, number fields, curves and
              cryptography},
    SERIES = {Math. Sci. Res. Inst. Publ.},
    VOLUME = {44},
     PAGES = {83--100},
 PUBLISHER = {Cambridge Univ. Press, Cambridge},
      YEAR = {2008},
      ISBN = {978-0-521-80854-5},
   MRCLASS = {11Y05 (11N36 11Y40)},
  MRNUMBER = {2467544},
MRREVIEWER = {Samuel\ S.\ Wagstaff, Jr.},
}

@article {Sun2000,
    AUTHOR = {Sun, Zhi-Hong},
     TITLE = {Congruences concerning {B}ernoulli numbers and {B}ernoulli
              polynomials},
   JOURNAL = {Discrete Appl. Math.},
  FJOURNAL = {Discrete Applied Mathematics. The Journal of Combinatorial
              Algorithms, Informatics and Computational Sciences},
    VOLUME = {105},
      YEAR = {2000},
    NUMBER = {1-3},
     PAGES = {193--223},
      ISSN = {0166-218X,1872-6771},
   MRCLASS = {11B68 (11A07)},
  MRNUMBER = {1780472},
MRREVIEWER = {Arnold\ M.\ Adelberg},
       DOI = {10.1016/S0166-218X(00)00184-0},
       URL = {https://doi.org/10.1016/S0166-218X(00)00184-0},
}

@article {Walsh2025,
    AUTHOR = {Walsh, P. G.},
     TITLE = {A question of {E}rd\H os on 3-powerful numbers and an elliptic
              curve analogue of the {A}nkeny-{A}rtin-{C}howla conjecture},
   JOURNAL = {Rad Hrvat. Akad. Znan. Umjet. Mat. Znan.},
  FJOURNAL = {Rad Hrvatske Akademije Znanosti i Umjetnosti. Matemati\v cke
              Znanosti},
    VOLUME = {29(564)},
      YEAR = {2025},
     PAGES = {83--87},
      ISSN = {1845-4100,1849-2215},
   MRCLASS = {11D25 (11G05)},
  MRNUMBER = {4849052},
}

@article {washington1976,
    AUTHOR = {Washington, L. C.},
     TITLE = {A note on {$p$}-adic {$L$}-functions},
   JOURNAL = {J. Number Theory},
  FJOURNAL = {Journal of Number Theory},
    VOLUME = {8},
      YEAR = {1976},
    NUMBER = {2},
     PAGES = {245--250},
      ISSN = {0022-314X,1096-1658},
   MRCLASS = {12A70},
  MRNUMBER = {406982},
MRREVIEWER = {Yasumasa\ Akagawa},
       DOI = {10.1016/0022-314X(76)90106-2},
       URL = {https://doi.org/10.1016/0022-314X(76)90106-2},
}

@book {Washington,
    AUTHOR = {Washington, Lawrence C.},
     TITLE = {Introduction to cyclotomic fields},
    SERIES = {Graduate Texts in Mathematics},
    VOLUME = {83},
 PUBLISHER = {Springer-Verlag, New York},
      YEAR = {1982},
      ISBN = {0-387-90622-3},
   MRCLASS = {11-01 (11R18 11R23)},
  MRNUMBER = {718674},
MRREVIEWER = {T.\ Mets\"ankyl\"a},
       DOI = {10.1007/978-1-4684-0133-2},
       URL = {https://doi.org/10.1007/978-1-4684-0133-2},
}

@incollection {Yu1998,
    AUTHOR = {Yu, Jing and Yu, Jiu-Kang},
     TITLE = {A note on a geometric analogue of {A}nkeny-{A}rtin-{C}howla's
              conjecture},
 BOOKTITLE = {Number theory ({T}iruchirapalli, 1996)},
    SERIES = {Contemp. Math.},
    VOLUME = {210},
     PAGES = {101--105},
 PUBLISHER = {Amer. Math. Soc., Providence, RI},
      YEAR = {1998},
      ISBN = {0-8218-0606-8},
   MRCLASS = {11R58},
  MRNUMBER = {1478488},
MRREVIEWER = {Franz\ Lemmermeyer},
       DOI = {10.1090/conm/210/02787},
       URL = {https://doi.org/10.1090/conm/210/02787},
}
\vspace*{1cm}
\textsc{Nic Fellini\newline 
Department of Mathematics and Statistics\newline 
Queen's University\newline  Kingston, ON, Canada  K7L 3N8}}\newline 
\texttt{n.fellini@queensu.ca}

\end{document}